\documentclass[11pt,twoside]{article}
\usepackage{amsmath, amsthm, amscd, amsfonts, amssymb, graphicx, color}

\date{}

\begin{document}

\centerline{}

\centerline {\Large{\bf Reflexivity of linear $n$-normed space with respect to}}
\centerline{\Large{\bf  $b$-linear functional}}

\newcommand{\mvec}[1]{\mbox{\bfseries\itshape #1}}
\centerline{}
\centerline{\textbf{Prasenjit Ghosh}}
\centerline{Department of Pure Mathematics, University of Calcutta,}
\centerline{35, Ballygunge Circular Road, Kolkata, 700019, West Bengal, India}
\centerline{e-mail: prasenjitpuremath@gmail.com}
\centerline{}
\centerline{\textbf{T. K. Samanta}}
\centerline{Department of Mathematics, Uluberia College,}
\centerline{Uluberia, Howrah, 711315,  West Bengal, India}
\centerline{e-mail: mumpu$_{-}$tapas5@yahoo.co.in}

\newtheorem{Theorem}{\quad Theorem}[section]

\newtheorem{definition}[Theorem]{\quad Definition}

\newtheorem{theorem}[Theorem]{\quad Theorem}

\newtheorem{remark}[Theorem]{\quad Remark}

\newtheorem{corollary}[Theorem]{\quad Corollary}

\newtheorem{note}[Theorem]{\quad Note}

\newtheorem{lemma}[Theorem]{\quad Lemma}

\newtheorem{example}[Theorem]{\quad Example}

\newtheorem{result}[Theorem]{\quad Result}
\newtheorem{conclusion}[Theorem]{\quad Conclusion}

\newtheorem{proposition}[Theorem]{\quad Proposition}

\begin{abstract}
\textbf{\emph{In continuation of the paper \cite{Prasenjit}, we discuss various consequences of Hahn-Banach theorem for bounded b-linear functional in linear n-normed space and describe the notion of reflexivity of linear n-normed space with respect to bounded b-linear functional.\,The concepts of strong convergence and weak convergence of a sequence of vectors with respect to bounded b-linear functionals in linear n-normed space have been introduced and some of their properties are being discussed.}}
\end{abstract}
{\bf Keywords:}  \emph{Hahn-Banach theorem, Reflexivity of normed  linear space, Weak \\ \smallskip\hspace{2cm}and strong convergence, Linear n-normed space, n-Banach space.}\\

\\

\section{Introduction}
 
\smallskip\hspace{.6 cm} The dual space of a normed linear space is the set of all bounded linear functionals on the space.\;In some cases, the dual of the dual space, i.\,e., second dual space of a normed space, under a specific mapping-called the natural embedding, is isometrically isomorphic to the original space.\;Such normed spaces are known as reflexive spaces.\;This concept was introduced by H.\,Hahn in 1927 and called reflexivity by E.\,R Lorch in 1939.\;Hahn recognized the importance of reflexivity in his study of linear equations in normed spaces.\;Weak convergence of sequence of vectors in a normed space is a certain kind of interplay between a normed space and its dual space.\;This concept demonstrates a fundamental principle of functional analysis which in turn states that the investigation of normed spaces is generally linked with that of their dual spaces.\;Weak convergence has various applications in the calculus of variations, general theory of differential equations and in fact, plays an important role in many problems of analysis.\\

The notion of linear\;$2$-normed space was introduced by S.\,Gahler \cite{Gahler}.\;A survey of the theory of linear\;$2$-normed space can be found in \cite{Freese}.\;The concept of $2$-Banach space is briefly discussed in \cite{White}.\;H.\,Gunawan and Mashadi \cite{Mashadi} developed the generalization of a linear\;$2$-normed space for \,$n \,\geq\, 2$.   

In this paper,\;some important consequences of the Hahn-Banach theorem for bounded \,$b$-linear functionals in case of linear\;$n$-normed spaces are discussed.\;We shall introduce the notion of \,$b$-relexivity of linear\;$n$-normed space and see that a closed subspce of a \,$b$-reflexive\;$n$-Banach space is also \,$b$-reflexive.\;Finally,\;$b$-weak convergence and \,$b$-strong convergence of a sequence of vectors in a linear\;$n$-normed space in terms of bounded \,$b$-linear functionals are introduced and characterized.

\section{Preliminaries}

\begin{theorem}\cite{E}\label{th0.01}
Let \,$\left\{\,T_{k}\,\right\}$\, be a sequence of bounded linear operators \,$T_{k} \,:\, Y \,\to\, Z$\, from a Banach space \,$Y$\, into a normed space \,$Z$\, such that \,$\left\{\,\left\|\,T_{k}\,(\,x\,)\,\right\|\,\right\}$\, is bounded for every \,$x \,\in\, Y$.\;Then the sequence of the norms \,$\left\{\,\left\|\,T_{k}\,\right\|\,\right\}$\, is bounded.  
\end{theorem}

\begin{definition}\cite{Mashadi}
Let \,$X$\, be a linear space over the field \,$ \mathbb{K}$, where \,$ \mathbb{K} $\, is the real or complex numbers field with \,$\text{dim}\,X \,\geq\, n$, where \,$n$\, is a positive integer.\;A real valued function \,$\left \|\,\cdot \,,\, \cdots \,,\, \cdot \,\right \| \,:\, X^{\,n} \,\to\, \mathbb{R}$\, is called an n-norm on \,$X$\, if
\begin{itemize}
\item[(N1)]$\left\|\,x_{\,1} \,,\, x_{\,2} \,,\, \cdots \,,\, x_{\,n}\,\right\| \,=\,0$\, if and only if \,$x_{\,1},\, \cdots,\, x_{\,n}$\, are linearly dependent,
\item[(N2)]$\left\|\,x_{\,1} \,,\, x_{\,2} \,,\, \cdots \,,\, x_{\,n}\,\right\|$\; is invariant under permutations of \,$x_{\,1},\, x_{\,2},\, \cdots,\, x_{\,n}$,
\item[(N3)]$\left\|\,\alpha\,x_{\,1} \,,\, x_{\,2} \,,\, \cdots \,,\, x_{\,n}\,\right\| \,=\, |\,\alpha\,|\, \left\|\,x_{\,1} \,,\, x_{\,2} \,,\, \cdots \,,\, x_{\,n}\,\right\|\; \;\;\forall \;\; \alpha \,\in\, \mathbb{K}$,
\item[(N4)]$\left\|\,x \,+\, y \,,\, x_{\,2} \,,\, \cdots \,,\, x_{\,n}\,\right\| \,\leq\, \left\|\,x \,,\, x_{\,2} \,,\, \cdots \,,\, x_{\,n}\,\right\| \,+\,  \left\|\,y \,,\, x_{\,2} \,,\, \cdots \,,\, x_{\,n}\,\right\|$
\end{itemize}
hold for all \,$x,\, y,\, x_{\,1},\, x_{\,2},\, \cdots,\, x_{\,n} \,\in\, X$.\;The pair \,$\left(\,X \,,\, \left \|\,\cdot \,,\, \cdots \,,\, \cdot \,\right \| \,\right)$\; is then called a linear n-normed space.\;For particular value \,$n \,=\, 2$, the space \,$X$\, is said to be a linear 2-normed space \cite{Gahler}. 
\end{definition}

Throughout this paper, \,$X$\, will denote linear\;$n$-normed space over the field \,$\mathbb{K}$\, (\,$= \,\mathbb{R}\, \;\text{or}\, \;\mathbb{C}$\,) associated with the $n$-norm \,$\|\,\cdot \,,\, \cdots \,,\, \cdot\,\|$.

\begin{definition}\cite{Mashadi}
A sequence \,$\{\,x_{\,k}\,\} \,\subseteq\, X$\, is said to converge to \,$x \,\in\, X$\; if 
\[\lim\limits_{k \to \infty}\,\left\|\,x_{\,k} \,-\, x \,,\, x_{\,2} \,,\, \cdots \,,\, x_{\,n} \,\right\| \,=\, 0\]
for every \,$ x_{\,2},\, \cdots,\, x_{\,n} \,\in\, X$\, and it is called a Cauchy sequence if 
\[\lim\limits_{l \,,\, k \to \infty}\,\left\|\,x_{\,l} \,-\, x_{\,k} \,,\, x_{\,2} \,,\, \cdots \,,\, x_{\,n}\,\right\| \,=\, 0\]
for every \,$ x_{\,2},\, \cdots,\, x_{\,n} \,\in\, X$.\;The space \,$X$\, is said to be complete or n-Banach space if every Cauchy sequence in this space is convergent in \,$X$.\;2-Banach space \cite{White} is a particular case of n-Banach space for \,$n \,=\, 2$. 
\end{definition}

\begin{definition}\cite{Soenjaya}
We define the following open and closed ball in \,$X$: 
\[B_{\,\{\,e_{\,2} \,,\, \cdots \,,\, e_{\,n}\,\}}\,(\,a \,,\, \delta\,) \,=\, \left\{\,x \,\in\, X \,:\, \left\|\,x \,-\, a \,,\, e_{\,2} \,,\, \cdots \,,\, e_{\,n}\,\right\| \,<\, \delta \,\right\}\;\text{and}\]
\[B_{\,\{\,e_{\,2} \,,\, \cdots \,,\, e_{\,n}\,\}}\,[\,a \,,\, \delta\,] \,=\, \left\{\,x \,\in\, X \,:\, \left\|\,x \,-\, a \,,\, e_{\,2} \,,\, \cdots \,,\, e_{\,n}\,\right\| \,\leq\, \delta\,\right\},\hspace{.5cm}\]
where \,$a,\, e_{\,2},\, \cdots,\, e_{\,n} \,\in\, X$\, and \,$\delta$\, be a positive number.
\end{definition}

\begin{definition}\cite{Soenjaya}
A subset \,$G$\, of \,$X$\, is said to be open in \,$X$\, if for all \,$a \,\in\, G $, there exist \,$e_{\,2},\, \cdots,\, e_{\,n} \,\in\, X $\, and \, $\delta \,>\, 0 $\; such that \,$B_{\,\{\,e_{\,2} \,,\, \cdots \,,\, e_{\,n}\,\}}\,(\,a \,,\, \delta\,) \,\subseteq\, G$.
\end{definition}

\begin{definition}\cite{Soenjaya}
Let \,$ A \,\subseteq\, X$.\;Then the closure of \,$A$\, is defined as 
\[\overline{A} \,=\, \left\{\, x \,\in\, X \;|\; \,\exists\, \;\{\,x_{\,k}\,\} \,\in\, A \;\;\textit{with}\;  \lim\limits_{k \,\to\, \infty} x_{\,k} \,=\, x \,\right\}.\]
The set \,$ A $\, is said to be closed if $ A \,=\, \overline{A}$. 
\end{definition}

\begin{definition}\cite{Prasenjit}
Let \,$W$\, be a subspace of \,$X$\, and \,$b_{\,2},\, b_{\,3},\, \cdots,\, b_{\,n}$\, be fixed elements in \,$X$\, and \,$\left<\,b_{\,i}\,\right>$\, denote the subspaces of \,$X$\, generated by \,$b_{\,i}$, for \,$i \,=\, 2,\, 3,\, \cdots,\,n $.\;Then a map \,$T \,:\, W \,\times\,\left<\,b_{\,2}\,\right> \,\times\, \cdots \,\times\, \left<\,b_{\,n}\,\right> \,\to\, \mathbb{K}$\; is called a b-linear functional on \,$W \,\times\, \left<\,b_{\,2}\,\right> \,\times\, \cdots \,\times\, \left<\,b_{\,n}\,\right>$, if for every \,$x,\, y \,\in\, W$\, and \,$k \,\in\, \mathbb{K}$, the following conditions hold:
\begin{itemize}
\item[(I)] $T\,(\,x \,+\, y \,,\, b_{\,2}  \,,\, \cdots \,,\, b_{\,n}\,) \,=\, T\,(\,x  \,,\, b_{\,2} \,,\, \cdots \,,\, b_{\,n}\,) \,+\, T\,(\,y  \,,\, b_{\,2} \,,\, \cdots \,,\, b_{\,n}\,)$
\item[(II)] $T\,(\,k\,x  \,,\, b_{\,2} \,,\, \cdots \,,\, b_{\,n}\,) \,=\, k\; T\,(\,x  \,,\, b_{\,2} \,,\, \cdots \,,\, b_{\,n}\,)$. 
\end{itemize}
A b-linear functional is said to be bounded if \,$\exists$\, a real number \,$M \,>\, 0$\, such that
\[\left|\,T\,(\,x  \,,\, b_{\,2} \,,\, \cdots \,,\, b_{\,n}\,)\,\right| \,\leq\, M\; \left\|\,x  \,,\, b_{\,2} \,,\, \cdots \,,\, b_{\,n}\,\right\|\; \;\forall\; x \,\in\, W.\]
The norm of the bounded b-linear functional \,$T$\, is defined by
\[\|\,T\,\| \,=\, \inf\,\left\{\,M \,>\, 0 \,:\, \left|\,T\,(\,x  \,,\, b_{\,2} \,,\, \cdots \,,\, b_{\,n}\,)\,\right| \,\leq\, M\; \left\|\,x  \,,\, b_{\,2} \,,\, \cdots \,,\, b_{\,n}\,\right\| \;\forall\; x \,\in\, W\,\right\}.\]
The norm of \,$T$\, can be expressed by any one of the following equivalent formula:
\begin{itemize}
\item[(I)]\hspace{.2cm}$\|\,T\,\| \,=\, \sup\,\left\{\,\left|\,T\,(\,x \,,\, b_{\,2} \,,\, \cdots \,,\, b_{\,n}\,)\,\right| \;:\; \left\|\,x  \,,\, b_{\,2} \,,\, \cdots \,,\, b_{\,n}\,\right\| \,\leq\, 1\,\right\}$.
\item[(II)]\hspace{.2cm}$\|\,T\,\| \,=\, \sup\,\left\{\,\left|\,T\,(\,x \,,\, b_{\,2} \,,\, \cdots \,,\, b_{\,n}\,)\,\right| \;:\; \left\|\,x  \,,\, b_{\,2} \,,\, \cdots \,,\, b_{\,n}\,\right\| \,=\, 1\,\right\}$.
\item[(III)]\hspace{.2cm}$ \|\,T\,\| \,=\, \sup\,\left \{\,\dfrac{\left|\,T\,(\,x \,,\, b_{\,2} \,,\, \cdots \,,\, b_{\,n}\,)\,\right|}{\left\|\,x \,,\, b_{\,2} \,,\, \cdots \,,\, b_{\,n}\,\right\|} \;:\; \left\|\,x  \,,\, b_{\,2} \,,\, \cdots \,,\, b_{\,n}\,\right\| \,\neq\, 0\,\right \}$. 
\end{itemize}
Also, we have \,$\left|\,T\,(\,x \,,\, b_{\,2} \,,\, \cdots \,,\, b_{\,n}\,)\,\right| \,\leq\, \|\,T\,\|\, \left\|\,x  \,,\, b_{\,2} \,,\, \cdots \,,\, b_{\,n}\,\right\|\, \;\forall\; x \,\in\, W$.
\end{definition}
Let \,$X_{F}^{\,\ast}$\, denote the Banach space of all bounded b-linear functional defined on \,$X \,\times\, \left<\,b_{\,2}\,\right> \,\times \cdots \,\times\, \left<\,b_{\,n}\,\right>$\, with respect to the above norm.

\begin{definition}\cite{Prasenjit}
A set \,$\mathcal{A}$\, of bounded b-linear functionals defined on  \,$X \,\times\, \left<\,b_{\,2}\,\right> \,\times \cdots \,\times\, \left<\,b_{\,n}\,\right>$\, is said to be pointwise bounded if for each \,$x \,\in\, X $, the set \,$\left\{\,T\,(\,x \,,\, b_{\,2} \,,\, \cdots \,,\, b_{\,n}\,) \,:\, T \,\in\, \;\mathcal{A} \,\right\}$\; is a bounded set in \;$\mathbb{K}$\, and  uniformly bounded if, \,$\exists\, \,K \,>\, 0$ such that \,$\|\,T\,\| \,\leq\, K\;\; \;\forall\; T \,\in\, \mathcal{A}$.
\end{definition}

\begin{theorem}\cite{Prasenjit}\label{th1}
Let \,$X$\, be a n-Banach space over the field \,$\mathbb{K}$.\;If a set \,$\mathcal{A}$\, of bounded b-linear functionals on \,$ X \,\times\, \left<\,b_{\,2}\,\right> \,\times\, \cdots \,\times\, \left<\,b_{\,n}\,\right>$\, is pointwise bounded, then it is uniformly bounded.
\end{theorem} 
 
\begin{theorem}\cite{Prasenjit}\label{th1.1}
Let \,$X$\, be a linear n-normed space over the field \,$\mathbb{R}$\, and \,$W$\, be a subspace of $\,X$.\;Then each bounded b-linear functional \,$T_{\,W}$\, defined on \,$W \,\times\, \left<\,b_{\,2}\,\right> \,\times\, \cdots \,\times\, \left<\,b_{\,n}\,\right>$\, can be extended onto \,$X \,\times\, \left<\,b_{\,2}\,\right> \,\times\, \cdots \,\times\, \left<\,b_{\,n}\,\right>$\, with preservation of the norm.\;In other words, there exists a bounded b-linear functional \,$T$\, defined on \,$X \,\times\, \left<\,b_{\,2}\,\right> \,\times\, \cdots \,\times\, \left<\,b_{\,n}\,\right>$\, such that
\[T\,(\,x \,,\, b_{\,2} \,,\, \cdots \,,\, b_{\,n}\,) \,=\, T_{\,W}\,(\,x \,,\, b_{\,2} \,,\, \cdots \,,\, b_{\,n}\,)\; \;\forall\; x \,\in\, W\,  \;\;\&\;\; \left\|\,T_{\,W}\,\right\| \,=\, \|\,T\,\|.\] 
\end{theorem}

\begin{theorem}\cite{Prasenjit}\label{th1.2}
Let \,$X$\, be a linear n-normed space over the field \,$\mathbb{R}$\, and \,$x_{\,0}$\, be an arbitrary non-zero element in \,$X$.\;Then there exists a bounded b-linear functional \,$T$\, defined on \,$X \,\times\, \left<\,b_{\,2}\,\right> \,\times\, \cdots \,\times\, \left<\,b_{\,n}\,\right>$\, such that 
\[ \|\,T\,\| \,=\, 1\; \;\;\text{and}\;\;  \;T\,(\,x_{\,0} \,,\, b_{\,2} \,,\, \cdots \,,\, b_{\,n}\,) \,=\, \left\|\,x_{\,0} \,,\, b_{\,2} \,,\, \cdots \,,\, b_{\,n}\,\right\|.\]
\end{theorem}

\begin{theorem}\cite{Prasenjit}\label{th1.3}
Let \,$X$\, be a linear n-normed space over the field \,$\mathbb{R}$\, and \,$\,x \,\in\, X$.\;Then 
\[\left\|\,x \,,\, b_{\,2} \,,\, \cdots \,,\, b_{\,n}\,\right\| \,=\, \sup\,\left\{\, \dfrac{\left|\,T\,(\,x \,,\, b_{\,2} \,,\, \cdots \,,\, b_{\,n}\,)\,\right|}{\|\,T\,\|} \,:\, T \,\in\, X^{\,\ast}_{F} \;,\; T \,\neq\, 0 \,\right\}.\]
\end{theorem}

\section{Consequences of Hahn-Banach theorem in linear\;$n$-normed space }

\smallskip\hspace{.6 cm} In this section, we shall consider some immediate corollaries and important consequences of the Hahn-Banach extension theorem for bounded \,$b$-linear functional \cite{Prasenjit} in case of linear\;$n$-normed space.

\begin{theorem}\label{th2}
Let \,$X$\, be a linear n-normed space over the field \,$\mathbb{R}$\, and let \,$x,\, y$\, be two distinct points of \,$X$\, such that the set \,$\left\{\,x,\, b_{\,2},\, \cdots,\, b_{\,n}\,\right\}$\, or \,$\left\{\,y,\, b_{\,2},\, \cdots,\, b_{\,n}\,\right\}$\, are linearly independent.\;Then \,$\exists$\, \,$T \,\in\, X_{F}^{\,\ast}$\, such that 
\[ T\,(\,x \,,\, b_{\,2} \,,\, \cdots \,,\, b_{\,n}\,) \,\neq\, T\,(\,y \,,\, b_{\,2} \,,\, \cdots \,,\, b_{\,n}\,).\]
\end{theorem}

\begin{proof}
Consider, \,$z \,=\, x \,-\, y$.\;Then \,$\theta \,\neq\, z \,\in\, X$\, and therefore by Theorem (\ref{th1.2}), \,$\exists\, \;T \,\in\, X^{\,\ast}_{F}$\; such that 
\[T\,(\,z \,,\, b_{\,2} \,,\, \cdots \,,\, b_{\,n}\,) \,=\, \left\|\,z \,,\, b_{\,2} \,,\, \cdots \,,\, b_{\,n}\,\right\|\; \;\text{and}\;\; \|\,T\,\| \,=\, 1\]
\[\Rightarrow\; T\,(\,x \,-\, y \,,\, b_{\,2} \,,\, \cdots \,,\, b_{\,n}\,) \,=\, \left\|\,x \,-\, y \,,\, b_{\,2} \,,\, \cdots \,,\, b_{\,n}\,\right\| \,\neq\, 0\hspace{.5cm}\]
\[\Rightarrow\; T\,(\,x \,,\, b_{\,2} \,,\, \cdots \,,\, b_{\,n}\,) \,-\, T\,(\,y \,,\, b_{\,2} \,,\, \cdots \,,\, b_{\,n}\,) \,\neq\, 0\hspace{2cm}\] 
\[\Rightarrow\; T\,(\,x \,,\, b_{\,2} \,,\, \cdots \,,\, b_{\,n}\,) \,\neq\, T\,(\,y \,,\, b_{\,2} \,,\, \cdots \,,\, b_{\,n}\,).\hspace{2.64cm}\]
\end{proof}

\begin{corollary}\label{cor1}
If \,$X \,\neq\, \{\,\theta\,\}$\, is a linear n-normed space, then there are always non-trivial bounded b-linear functionals on \,$X \,\times\, \left<\,b_{\,2}\,\right> \,\times\, \cdots \,\times\, \left<\,b_{\,n}\,\right>$, i.\,e., \,$X \,\neq\, \{\,\theta\,\} \,\Rightarrow\, X^{\,\ast}_{F} \,\neq\, \{\,O\,\},\, O$\, being a null operator.
\end{corollary}

\begin{proof}
This is an immediate consequence of Theorem (\ref{th1.2}). 
\end{proof}

\begin{corollary}\label{cor2}
Let \,$X$\, be a linear n-normed space.\;Then for all \,$T \,\in\, X^{\,\ast}_{F}$,
\[T\,(\,x \,,\, b_{\,2} \,,\, \cdots \,,\, b_{\,n}\,) \,=\, 0\; \,\Rightarrow\,  x \,=\, \theta.\]
\end{corollary}

\begin{proof}
If possible let \,$x \,\neq\, \theta$.\;Then by the Corollary (\ref{cor1}), \,$\exists\, \,T \,\in\, X^{\,\ast}_{F}$\; such that \,$T\,(\,x \,,\, b_{\,2} \,,\, \cdots \,,\, b_{\,n}\,) \,\neq\, 0$.\;This is a contradiction to the given hypothesis.\;Hence the results follows.  
\end{proof}

We now proceed to present another implication of the Hahn-Banach theorem for bounded \,$b$-linear functional and establish that there are always sufficient bounded \,$b$-linear functionals on a linear $n$-normed space which separate points from proper subspaces. 

\begin{theorem}\label{th3.1}
Let \,$X$\; be a linear n-normed space over the field \,$\mathbb{R}$\, and \,$W$\, be a subspace of \,$X$\, and let \,$x_{\,0} \,\in\, X$\, such that \,$x_{\,0},\, b_{\,2},\, \cdots,\, b_{\,n}$\, are linearly independent and suppose \,$d \,=\, \inf\limits_{x \,\in\, W}\,\left\|\,x_{\,0} \,-\, x \,,\, b_{\,2} \,,\, \cdots \,,\, b_{\,n}\,\right\| \,>\, 0$.\;Then \,$\exists$\, \,$T \,\in\, X^{\,\ast}_{F}$\, such that
\begin{itemize}
\item[(I)]\hspace{.1cm} $T\,\left(\,x_{\,0} \,,\, b_{\,2} \,,\, \cdots \,,\, b_{\,n}\,\right) \,=\, 1$,  
\item[(II)]\hspace{.1cm} $T\,\left(\,x \,,\, b_{\,2} \,,\, \cdots \,,\, b_{\,n}\,\right) \,=\, 0\; \;\;\forall\; x \,\in\, W\;\; \;\text{and}\;\;  \;\|\,T\,\| \,=\, \dfrac{1}{d}$.
\end{itemize}
\end{theorem}

\begin{proof}
Let \,$W_{\,0} \,=\, W \,+\, \left<\,x_{\,0}\,\right>$\; be the space spannded by \,$W$\, and \,$x_{\,0}$.\;Since \,$d \,>\, 0 $, we have \,$x_{\,0} \,\not\in W$.\;Therefore, each \,$x \,\in\, W_{0}$\, can be expressed uniquely in the form \,$x \,=\, y \,+\, \alpha\,x_{\,0},\; y \,\in\, W$\, and \,$\alpha \,\in\, \mathbb{R}$.\;We define a functional as follows: 
\[T_{\,1} \,:\, W_{0} \,\times\, \left<\,b_{\,2}\,\right> \,\times \cdots \,\times\, \left<\,b_{\,n}\,\right> \,\to\, \mathbb{R},\; T_{\,1}\,\left(\,y \,+\, \alpha\,x_{\,0} \,,\, b_{\,2} \,,\, \cdots \,,\, b_{\,n}\,\right) \,=\, \alpha.\]Then clearly \,$T_{\,1}$\, is a \,$b$-linear functional on \,$W_{0} \,\times\, \left<\,b_{\,2}\,\right> \,\times \cdots \,\times\, \left<\,b_{\,n}\,\right>$\, satisfying 
\[T_{\,1}\,\left(\,x \,,\, b_{\,2} \,,\, \cdots \,,\, b_{\,n}\,\right) \,=\, 0\; \;\;\forall\; x \,\in\, W\; \,\;\text{and}\,\; \,T_{\,1}\,\left(\,x_{\,0} \,,\, b_{\,2} \,,\, \cdots \,,\, b_{\,n}\,\right) \,=\, 1.\]Also, for each \,$x \,\in\, W_{0}$, we have
\[\left|\,T_{\,1}\,\left(\,x \,,\, b_{\,2} \,,\, \cdots \,,\, b_{\,n}\,\right)\,\right| \,=\, \left|\,T_{\,1}\,\left(\,y \,+\, \alpha\,x_{\,0} \,,\, b_{\,2} \,,\, \cdots \,,\, b_{\,n}\,\right)\,\right| \,=\, \,|\,\alpha\,| \hspace{4cm}\]
\[\,=\, \dfrac{|\,\alpha\,|\,\left\|\,x,\, b_{\,2},\, \cdots,\, b_{\,n}\,\right\|}{\left\|\,x,\, b_{\,2},\, \cdots,\, b_{\,n}\,\right\|} =\, \dfrac{|\,\alpha\,|\,\left\|\,x, b_{\,2},\, \cdots,\, b_{\,n}\,\right\|}{\left\|\,y \,+\, \alpha\,x_{\,0},\, b_{\,2},\, \cdots,\, b_{\,n}\,\right\|} \,=\,  \dfrac{|\,\alpha\,|\,\left\|\,x,\, b_{\,2},\, \cdots,\, b_{\,n}\,\right\|}{|\,\alpha\,|\,\left\|\,\dfrac{y}{\alpha} \,+\, x_{\,0},\, b_{\,2},\, \cdots,\, b_{\,n}\,\right\|}\]
\[ \,=\, \dfrac{\left\|\,x,\, b_{\,2},\, \cdots,\, b_{\,n}\,\right\|}{\left\|\,x_{\,0} \,-\, \left(\,-\, \dfrac{y}{\alpha}\,\right) \,,\, b_{\,2} \,,\, \cdots \,,\, b_{\,n}\,\right\|}\,\leq\, \dfrac{\left\|\,x,\, b_{\,2},\, \cdots,\, b_{\,n}\,\right\|}{d}\; \;\left[\;\text{since}\; \,-\, \dfrac{y}{\alpha} \,\in\, W\right].\hspace{1cm}\]This shows that \,$T_{\,1}$\, is a bounded \,$b$-linear functional with \,$\|\,T_{\,1}\,\| \,\leq\, \dfrac{1}{d}$.\;To prove \,$\|\,T_{\,1}\,\| \,\geq\, \dfrac{1}{d}$, we consider a sequence \,$\left\{\,x_{\,k}\,\right\},\, x_{\,k} \,\in\, W$\, such that 
\[\lim\limits_{k \,\to\, \infty}\,\left\|\,x_{\,0} \,-\, x_{\,k} \,,\, b_{\,2} \,,\, \cdots \,,\, b_{\,n}\,\right\| \,=\, d.\]
\[\text{Now},\hspace{.3cm} 1 \,=\, \left|\,T_{\,1}\,\left(\,x_{\,0} \,,\, b_{\,2} \,,\, \cdots \,,\, b_{\,n}\,\right) \,-\, T_{\,1}\,\left(\,x_{\,k} \,,\, b_{\,2} \,,\, \cdots \,,\, b_{\,n}\,\right)\,\right|\hspace{2cm}\]
\[\hspace{1cm} \,=\, \left|\,T_{\,1}\,\left(\,x_{\,0} \,-\, x_{\,k} \,,\, b_{\,2} \,,\, \cdots \,,\, b_{\,n}\,\right) \,\right| \,\leq\, \|\,T_{\,1}\,\|\, \left\|\,x_{\,0} \,-\, x_{\,k} \,,\, b_{\,2} \,,\, \cdots \,,\, b_{\,n}\,\right\|\]
\[\Rightarrow\, 1 \,\leq\, \|\,T_{\,1}\,\|\, \lim\limits_{k \,\to\, \infty}\,\left\|\,x_{\,0} \,-\, x_{\,k} \,,\, b_{\,2} \,,\, \cdots \,,\, b_{\,n}\,\right\| \,=\, \|\,T_{\,1}\,\|\; d \,\Rightarrow\, \|\,T_{\,1}\,\| \,\geq\, \dfrac{1}{d}\;.\]Thus, we have established that \,$\exists$\, a bounded \,$b$-linear functional \,$T_{\,1}$\, on \,$W_{0} \,\times\, \left<\,b_{\,2}\,\right> \,\times \cdots \,\times\, \left<\,b_{\,n}\,\right>$\, such that 
\[T_{\,1}\,\left(\,x \,,\, b_{\,2} \,,\, \cdots \,,\, b_{\,n}\,\right) \,=\, 0\; \;\;\forall\; x \,\in\, W,\; \;T_{\,1}\,\left(\,x_{\,0} \,,\, b_{\,2} \,,\, \cdots \,,\, b_{\,n}\,\right) \,=\, 1\; \;\text{and}\; \,\|\,T_{\,1}\,\| \,=\, \dfrac{1}{d}.\]
Applying the Theorem (\ref{th1.1}), we obtain a \,$b$-linear functional \,$T \,\in\, X^{\,\ast}_{F}$\, such that 
\[T\,\left(\,x \,,\, b_{\,2} \,,\, \cdots \,,\, b_{\,n}\,\right) \,=\, T_{\,1}\,\left(\,x \,,\, b_{\,2} \,,\, \cdots \,,\, b_{\,n}\,\right)\; \;\forall\; x \,\in\, W_{0}\; \;\text{and}\; \,\|\,T\,\| \,=\, \|\,T_{\,1}\,\| \,=\, \dfrac{1}{d}.\]
\[\text{So},\hspace{.3cm} T\,\left(\,x \,,\, b_{\,2} \,,\, \cdots \,,\, b_{\,n}\,\right) \,=\, T_{\,1}\,\left(\,x \,,\, b_{\,2} \,,\, \cdots \,,\, b_{\,n}\,\right) \,=\, 0\; \;\;\forall\; x \,\in\, W\; \;\text{and}\] 
\[T\,\left(\,x_{\,0} \,,\, b_{\,2} \,,\, \cdots \,,\, b_{\,n}\,\right) \,=\, T_{\,1}\,\left(\,x_{\,0} \,,\, b_{\,2} \,,\, \cdots \,,\, b_{\,n}\,\right) \,=\, 1.\]Hence, the proof of the theorem is complete.
\end{proof}

\begin{remark}
The Theorem (\ref{th3.1}) is a generalization of the Theorem (\ref{th1.2}) and its derivation is as follows:\\

Consider \,$W \,=\, \{\,0\,\}$\, and \,$d \,=\, \left\|\,x_{\,0} \,,\, b_{\,2} \,,\, \cdots \,,\, b_{\,n}\,\right\|$, then by the Theorem (\ref{th3.1}), there exists a bounded $b$-linear functional \,$T_{\,0} \,\in\, X^{\,\ast}_{F}$\, such that
\[\left\|\,T_{\,0}\,\right\| \,=\, \dfrac{1}{d} \,=\, \dfrac{1}{\left\|\,x_{\,0} \,,\, b_{\,2} \,,\, \cdots \,,\, b_{\,n}\,\right\|}\; \;\text{and}\; \;T_{\,0}\,(\,x_{\,0} \,,\, b_{\,2} \,,\, \cdots \,,\, b_{\,n}\,) \,=\, 1.\]
Now, for all \,$x \,\in\, X$, we define  
\[T\,(\,x \,,\, b_{\,2} \,,\, \cdots \,,\, b_{\,n}\,) \,=\, \left\|\,x_{\,0} \,,\, b_{\,2} \,,\, \cdots \,,\, b_{\,n}\,\right\| \,\cdot\, T_{\,0}\,(\,x \,,\, b_{\,2} \,,\, \cdots \,,\, b_{\,n}\,),\;\text{then}\] 
\[T\,(\,x_{\,0} \,,\, b_{\,2} \,,\, \cdots \,,\, b_{\,n}\,) \,=\, \left\|\,x_{\,0} \,,\, b_{\,2} \,,\, \cdots \,,\, b_{\,n}\,\right\|\;T_{\,0}\,(\,x_{\,0} \,,\, b_{\,2} \,,\, \cdots \,,\, b_{\,n}\,) \,=\, \left\|\,x_{\,0} \,,\, b_{\,2} \,,\, \cdots \,,\, b_{\,n}\,\right\|\]
\[\text{Also},\;\|\,T\,\| \,=\, \sup \left \{\,\dfrac{|\,T\,(\,x \,,\, b_{\,2} \,,\, \cdots \,,\, b_{\,n}\,)\,|}{\|\,x \,,\, b_{\,2} \,,\, \cdots \,,\, b_{\,n}\,\|} \,:\, \|\,x \,,\, b_{\,2} \,,\, \cdots \,,\, b_{\,n}\,\| \,\neq\, 0\,\right \}\]
\[ \,=\, \sup \left \{\,\dfrac{\left|\,\left\|\,x_{\,0} \,,\, b_{\,2} \,,\, \cdots \,,\, b_{\,n}\,\right\|\;T_{\,0}\,(\,x \,,\, b_{\,2} \,,\, \cdots \,,\, b_{\,n}\,)\,\right|}{\|\,x \,,\, b_{\,2} \,,\, \cdots \,,\, b_{\,n}\,\|} \,:\, \|\,x \,,\, b_{\,2} \,,\, \cdots \,,\, b_{\,n}\,\| \,\neq\, 0\,\right \}\]
\[\,=\, \left\|\,x_{\,0} \,,\, b_{\,2} \,,\, \cdots \,,\, b_{\,n}\,\right\|\;\sup \left \{\,\dfrac{|\,T_{\,0}\,(\,x \,,\, b_{\,2} \,,\, \cdots \,,\, b_{\,n}\,)\,|}{\|\,x \,,\, b_{\,2} \,,\, \cdots \,,\, b_{\,n}\,\|} \,:\, \|\,x \,,\, b_{\,2} \,,\, \cdots \,,\, b_{\,n}\,\| \,\neq\, 0\,\right \}\]
\[ \,=\, \left\|\,x_{\,0} \,,\, b_{\,2} \,,\, \cdots \,,\, b_{\,n}\,\right\|\;\|\,T_{\,0}\,\| \,=\, 1.\hspace{7.4cm}\]
\end{remark}

\begin{corollary}
Let \,$X$\; be a linear n-normed space over the field \,$\mathbb{R}$\, and \,$W$\, be a subspace of \,$X$\; and let \,$x_{\,0} \,\in\, X$\; such that \,$x_{\,0},\, b_{\,2},\, \cdots,\, b_{\,n}$\, are linearly independent and suppose \,$d \,=\, \inf\limits_{x \,\in\, W}\,\left\|\,x_{\,0} \,-\, x \,,\, b_{\,2} \,,\, \cdots \,,\, b_{\,n}\,\right\| \,>\, 0$.\;Then
\begin{itemize}
\item[(I)]\hspace{.1cm} $T\,\left(\,x_{\,0} \,,\, b_{\,2} \,,\, \cdots \,,\, b_{\,n}\,\right) \,=\, d$,  
\item[(II)] \hspace{.1cm}$T\,\left(\,x \,,\, b_{\,2} \,,\, \cdots \,,\, b_{\,n}\,\right) \,=\, 0\; \;\;\forall\; x \,\in\, W\;\; \;\&\;\;  \;\|\,T\,\| \,=\, 1$, for some \,$T \,\in\, X^{\,\ast}_{F}$.
\end{itemize}
\end{corollary}

\begin{proof}
By Theorem (\ref{th3.1}), \,$\exists$\, \,$T_{\,1} \,\in\, X^{\,\ast}_{F}$\, such that
\[T_{\,1}\,\left(\,x_{\,0} \,,\, b_{\,2} \,,\, \cdots \,,\, b_{\,n}\,\right) \,=\, 1, \;T_{\,1}\,\left(\,x \,,\, b_{\,2} \,,\, \cdots \,,\, b_{\,n}\,\right) \,=\, 0\; \;\;\forall\; x \,\in\, W\; \;\text{and}\; \,\|\,T_{\,1}\,\| \,=\, \dfrac{1}{d}.\]
Define the bounded \,$b$-linear functional \,$T$\, on \,$X \,\times\, \left<\,b_{\,2}\,\right> \,\times \cdots \,\times\, \left<\,b_{\,n}\,\right>$\, by \,$T \,=\, d\; T_{\,1}$.\;Then
\,$T\,\left(\,x_{\,0} \,,\,  b_{\,2} \,,\, \cdots \,,\, b_{\,n}\,\right) \,=\, d\; T_{\,1}\,\left(\,x_{\,0} \,,\,  b_{\,2} \,,\, \cdots \,,\, b_{\,n}\,\right) \,=\, d$,
\[ T\,\left(\,x \,,\,  b_{\,2} \,,\, \cdots \,,\, b_{\,n}\,\right) \,=\, d\; T_{\,1}\,\left(\,x \,,\,  b_{\,2} \,,\, \cdots \,,\, b_{\,n}\,\right) \,=\, 0\; \;\;\forall\; x \,\in\, W\]with \,$\|\,T\,\| \,=\, d\; \|\,T_{\,1}\,\| \,=\, \dfrac{d}{d} \,=\, 1$.\;This completes the proof.
\end{proof}

\begin{corollary}\label{cor3}
Let \,$X$\; be a linear n-normed space over the field \,$\mathbb{R}$\, and \,$W$\, be a closed linear subspace of \,$X$\, and let \,$x_{\,0} \,\in\, X \,-\, W$\, such that \,$x_{\,0},\, b_{\,2},\, \cdots,\, b_{\,n}$\, are linearly independent and suppose \,$d \,=\, \inf\limits_{x \,\in\, W}\,\left\|\,x_{\,0} \,-\, x \,,\, b_{\,2} \,,\, \cdots \,,\, b_{\,n}\,\right\|$.\;Then \,$\exists$\, \,$T \,\in\, X^{\,\ast}_{F}$\, such that
\begin{itemize}
\item[(I)]\hspace{.1cm} $T\,\left(\,x_{\,0} \,,\, b_{\,2} \,,\, \cdots \,,\, b_{\,n}\,\right) \,=\, 1$,  
\item[(II)]\hspace{.1cm} $T\,\left(\,x \,,\, b_{\,2} \,,\, \cdots \,,\, b_{\,n}\,\right) \,=\, 0\; \;\;\forall\; x \,\in\, W\; \;\text{and}\; \;\|\,T\,\| \,=\, \dfrac{1}{d}$.
\end{itemize}
\end{corollary}
\begin{proof}
It can be easily verified that \,$\inf\limits_{x \,\in\, W}\,\left\|\,x_{\,0} \,-\, x \,,\, b_{\,2} \,,\, \cdots \,,\, b_{\,n}\,\right\| \,=\, 0$\, if and only if \,$x_{\,0} \,\in\, \overline{\,W}$.\;But \,$W \,=\, \overline{\,W}$\, and it follows that \,$x_{\,0} \,\not\in\, \overline{\,W}$.\;Hence
\[d \,=\, \inf\limits_{x \,\in\, W}\,\left\|\,x_{\,0} \,-\, x \,,\, b_{\,2} \,,\, \cdots \,,\, b_{\,n}\,\right\| \,>\, 0.\]Now, the proof of this corollary follows from Theorem (\ref{th3.1}). 
\end{proof}

\begin{corollary}\label{cor4}
Let \,$X$\; be a linear n-normed space over the field \,$\mathbb{R}$\, and \,$W$\, be a closed linear subspace of \,$X$\, and let \,$x_{\,0} \,\in\, X \,-\, W$\, such that \,$x_{\,0},\, b_{\,2},\, \cdots,\, b_{\,n}$\, are linearly independent.\;Then \,$\exists$\, \,$T \,\in\, X^{\,\ast}_{F}$\, such that
\[T\,\left(\,x_{\,0} \,,\, b_{\,2} \,,\, \cdots \,,\, b_{\,n}\,\right) \,\neq\, 0\; \;\text{and}\; \;T\,\left(\,x \,,\, b_{\,2} \,,\, \cdots \,,\, b_{\,n}\,\right) \,=\, 0\; \;\;\forall\; x \,\in\, W.\]
\end{corollary}

\begin{proof}
Proof of this corollary directly follows from that of the corollary (\ref{cor3})\,.
\end{proof}

The Hahn-Banach Theorem for bounded\;$b$-linear functional and its consequences can be used to revel much among the properties of linear\;$n$-normed space and its dual space.\;Next theorem relates separability of the dual space to the separability of its original space.

\begin{theorem}
Let \,$X$\, be a linear n-normed space over the field \,$\mathbb{R}$\, and \,$X^{\,\ast}_{F}$\, be the Banach space of all bounded b-linear functionals defined on \,$X \,\times\, \left<\,b_{\,2}\,\right> \,\times \cdots \,\times\, \left<\,b_{\,n}\,\right>$.\;Then the space \,$X$\, is separable if \,$X^{\,\ast}_{F}$\; is separable.
\end{theorem}

\begin{proof}
Since \,$X^{\,\ast}_{F}$\; is separable, \,$\exists$\, a countable set \,$S \,=\, \left\{\;T_{k} \,\in\, X^{\,\ast}_{F} \,:\, k \,\in\, \mathbb{N}\;\right\}$\; such that \,$S$\; is dense in \,$X^{\,\ast}_{F}$, i.\,e., \,$\overline{S} \,=\, X^{\,\ast}_{F}$.\;For each \,$k \,\in\, \mathbb{N}$, choose \,$x_{\,k} \,\in\, X$\; such that \,$\left\|\,x_{\,k} \,,\, b_{\,2} \,,\, \cdots \,,\, b_{\,n}\,\right\| \,=\, 1$\, and \,$\left|\,T_{k} \left(\,x_{\,k} \,,\, b_{\,2} \,,\, \cdots \,,\, b_{\,n}\,\right)\,\right| \,\geq\, \dfrac{1}{2}\, \left\|\,T_{k}\,\right\|$.\;Let \,$W$\; be the closed subspace of \,$X$\; generated by the sequence \,$\left\{\,x_{\,k}\,\right\}_{k \,=\, 1}^{\,\infty}$, i.\,e., \,$W \,=\, \overline{\,span}\,\left\{\,x_{\,k} \,\in\, X \,:\, k \,\in\, \mathbb{N}\,\right\}$.\;Suppose \,$W \,\neq\, X$.\;Let \,$x_{\,0} \,\in\, X \,-\, W$\, such that \,$x_{\,0},\, b_{\,2},\,\\ \cdots,\, b_{\,n}$\, are linearly independent.\;By Corollary (\ref{cor4}), \,$\exists$\, \,$0 \,\neq\, T \,\in\, X^{\,\ast}_{F}$\; such that 
\[T\,\left(\,x_{\,0} \,,\, b_{\,2} \,,\, \cdots \,,\, b_{\,n}\,\right) \,\neq\, 0\; \;\text{and}\; \;T\,\left(\,x \,,\, b_{\,2} \,,\, \cdots \,,\, b_{\,n}\,\right) \,=\, 0\; \;\;\forall\; x \,\in\, W.\]Since \,$\left\{\,x_{\,k}\,\right\}_{k \,=\, 1}^{\,\infty} \,\subseteq\, W$, \,$T\,\left(\,x_{\,k} \,,\, b_{\,2} \,,\, \cdots \,,\, b_{\,n}\,\right) \,=\, 0,\; k \,\in\, \mathbb{N}$.\;Thus 
\[ \dfrac{1}{2}\, \left\|\,T_{k}\,\right\| \,\leq\, \left|\,T_{k} \left(\,x_{\,k} \,,\, b_{\,2} \,,\, \cdots \,,\, b_{\,n}\,\right)\,\right| \,=\, \left|\,T_{k} \left(\,x_{\,k} \,,\, b_{\,2} \,,\, \cdots \,,\, b_{\,n}\,\right) \,-\, T\,\left(\,x_{\,k} \,,\, b_{\,2} \,,\, \cdots \,,\, b_{\,n}\,\right)\,\right|\]
\[\hspace{3.5cm} \,\leq\, \left\|\,T_{k} \,-\, T\,\right\|\,\left\|\,x_{\,k} \,,\, b_{\,2} \,,\, \cdots \,,\, b_{\,n}\,\right\|\]
\[\hspace{5.9cm}=\, \left\|\,T_{k} \,-\, T\,\right\| \;[\;\text{since}\; \left\|\,x_{\,k} \,,\, b_{\,2} \,,\, \cdots \,,\, b_{\,n}\,\right\| \,=\, 1\;].\]
Again, since \,$\overline{\,S} \,=\, X^{\,\ast}_{F}$, for each \,$T \,\in\, X^{\,\ast}_{F}$, \,$\exists\, \;\text{a sequence}\; \,\left\{\,T_{k}\,\right\}$\, in \,$S$\, such that \,$\lim\limits_{k \,\to\, \infty}\,T_{k} \,=\, T$.\;Therefore,
\[\|\,T\,\| \,\leq\, \left\|\,T_{k} \,-\, T\,\right\| \,+\, \left\|\,T_{k}\,\right\| \,\leq\, 3\,\left\|\,T_{k} \,-\, T\,\right\|\; \;\forall\, k \,\in\, \mathbb{N}.\]Taking limit on both sides as \,$k \,\to\, \infty$, it follows that \,$T \,=\, 0$, which contradicts the assumption that \,$W \,\neq\, X$.\;Hence, \,$W \,=\, X$\, and thus \,$X$\, is separable.    
\end{proof}

\section{Reflexivity of linear $n$-normed space}

\smallskip\hspace{.6 cm} Recall that given  a linear\;$n$-normed space \,$X \,\neq\, \{\,0\,\}$, the dual space \,$X_{F}^{\,\ast}$\, is a normed space with respect to the norm \,$\|\,\cdot\,\| \,:\, X_{F}^{\,\ast} \,\to\, \mathbb{R}$\, defined by 
\[ \|\,T\,\| \,=\, \sup\,\left\{\,\left|\,T\,(\,x \,,\, b_{\,2} \,,\, \cdots \,,\, b_{\,n}\,)\,\right| \,:\, x \,\in\, X,\, \left\|\,x  \,,\, b_{\,2} \,,\, \cdots \,,\, b_{\,n}\,\right\| \,=\, 1\,\right\}.\]
Furthermore,  \,$X_{F}^{\,\ast}$\, is a Banach space.\;Also, by Corollary (\ref{cor1}), \,$X_{F}^{\,\ast} \,\neq\, \{\,O\,\}$\, and, therefore, as a normed space \,$X_{F}^{\,\ast}$\, has its own dual space \,$\left(\,X_{F}^{\,\ast}\,\right)^{\,\ast}$, denoted by \,$X_{F}^{\,\ast\,\ast}$\, and is called the second dual space of \,$X$, which is again a Banach space under the norm
\[\|\,\varphi\,\| \,=\, \sup\,\left\{\,\left|\,\varphi\,(\,T\,)\,\right|\, \;:\; T \,\in\, X_{F}^{\,\ast} \,,\, \|\,T\,\| \,\leq\, 1\,\right\} \,,\, \varphi \,\in\, X_{F}^{\,\ast\,\ast}.\]

\begin{theorem}
Let \,$X$\, be a real linear n-normed space.\;Given \,$x \,\in\, X$, let 
\begin{equation}\label{eq1}
\varphi_{(\,x \,,\, F\,)}\,(\,T\,) \,=\, T\,(\,x \,,\, b_{\,2} \,,\, \cdots \,,\, b_{\,n}\,)\;\; \;\forall\; T \,\in\, X_{F}^{\,\ast}.
\end{equation}
Then \,$\varphi_{(\,x \,,\, F\,)}$\, is a bounded linear functional on \,$X_{F}^{\,\ast}$.\;Furthermore, the mapping \\$\left(\,x \,,\, b_{\,2} \,,\, \cdots \,,\, b_{\,n}\,\right) \,\to\, \varphi_{(\,x \,,\, F\,)}$\; is an isometric isomorphism of  \,$X \,\times\, \left<\,b_{\,2}\,\right> \,\times \cdots \,\times\, \left<\,b_{\,n}\,\right>$\; onto the subspace \,$\left\{\,\varphi_{(\,x \,,\, F\,)} \,:\, (\,x \,,\, b_{\,2} \,,\, \cdots \,,\, b_{\,n}\,) \,\in\, X \,\times\, \left<\,b_{\,2}\,\right> \,\times \cdots \,\times\, \left<\,b_{\,n}\,\right>\,\right\}$\; of \,$X_{F}^{\,\ast\,\ast}$.    
\end{theorem}

\begin{proof}
Let \,$\alpha,\, \beta \,\in\, \mathbb{R}$.\;Then
\[\varphi_{(\,x \,,\, F\,)}\,\left(\,\alpha\,T_{\,1} \,+\, \beta\,T_{\,2}\,\right) \,=\, \left(\,\alpha\,T_{\,1} \,+\, \beta\,T_{\,2}\,\right)\,(\,x \,,\, b_{\,2} \,,\, \cdots \,,\, b_{\,n}\,)\hspace{2.2cm}\]
\[\hspace{4.3cm}=\,\alpha\, T_{\,1}\,(\,x \,,\, b_{\,2} \,,\, \cdots \,,\, b_{\,n}\,) \,+\, \beta\, T_{\,2}\,(\,x \,,\, b_{\,2} \,,\, \cdots \,,\, b_{\,n}\,)\]
\[\hspace{4.8cm}=\,\alpha\,\varphi_{(\,x \,,\, F\,)}\,(\,T_{\,1}\,) \,+\, \beta\,\varphi_{(\,x \,,\, F\,)}\,(\,T_{\,2}\,)\; \;\forall\; T_{\,1} \,,\, T_{\,2} \,\in\, X_{F}^{\,\ast}.\]
So, \,$\varphi_{(\,x \,,\, F\,)}$\; is linear functional.\;Also, for all \,$T \,\in\, X_{F}^{\,\ast}$, we have
\[\left|\,\varphi_{(\,x \,,\, F\,)}\,(\,T\,)\,\right| \,=\, \left|\,T\,\left(\,x \,,\, b_{\,2} \,,\, \cdots \,,\, b_{\,n}\,\right)\,\right| \,\leq\, \left\|\,x \,,\, b_{\,2} \,,\, \cdots \,,\, b_{\,n}\,\right\|\, \|\,T\,\|.\] Consequently, \,$\varphi_{(\,x \,,\, F\,)} \,\in\, X_{F}^{\,\ast\,\ast}$\, with \,$\left\|\,\varphi_{(\,x \,,\, F\,)}\,\right\| \,\leq\, \left\|\,x \,,\, b_{\,2} \,,\, \cdots \,,\, b_{\,n}\,\right\|$.\;Moreover, such \,$\varphi_{(\,x \,,\, F\,)}$\; is unique.\;So, for every fixed \,$x \,\in\, X$\, there corresponds a unique bounded linear functional \,$\varphi_{(\,x \,,\, F\,)} \,\in\, X_{F}^{\,\ast\,\ast}$\; given by (\ref{eq1}).\;This defines a function \,$J \,:\, X \,\times\, \left<\,b_{\,2}\,\right> \,\times \cdots \,\times\, \left<\,b_{\,n}\,\right> \,\to\, X_{F}^{\,\ast\,\ast}$\, given by \,$J\,\left(\,x \,,\, b_{\,2} \,,\, \cdots \,,\, b_{\,n}\,\right) \,=\, \varphi_{(\,x \,,\, F\,)}$.\;We now verify that \,$J$\, is an isomorphism between \,$X \,\times\, \left<\,b_{\,2}\,\right> \,\times \cdots \,\times\, \left<\,b_{\,n}\,\right>$\, and the range of \,$J$, which is a subspace of \,$X_{F}^{\,\ast\,\ast}$.
\begin{itemize}
\item[(I)]\hspace{.1cm}
Let \,$x,\, y \,\in\, X\; \;\text{and}\; \;\alpha,\, \beta \,\in\, \mathbb{R}$.\;Then for all \,$T \,\in\, X_{F}^{\,\ast}$, we have
\[\left[\,J\,\left(\,\alpha\,x \,+\, \beta\,y \,,\, b_{\,2} \,,\, \cdots \,,\, b_{\,n}\,\right)\,\right]\,(\,T\,) \,=\, \varphi_{(\,\alpha\,x \,+\, \beta\,y  \,,\, F\,)}\,(\,T\,)\]
\[=\, T\,\left(\,\alpha\,x \,+\, \beta\,y \,,\, b_{\,2} \,,\, \cdots \,,\, b_{\,n}\,\right) \,=\, \alpha\; T\,\left(\,x \,,\, b_{\,2} \,,\, \cdots \,,\, b_{\,n}\,\right) \,+\, \beta\; T\,\left(\,y \,,\, b_{\,2} \,,\, \cdots \,,\, b_{\,n}\,\right)\]
\[=\, \alpha\; \varphi_{(\,x \,,\, F\,)}\,(\,T\,) \,+\, \beta\; \varphi_{(\,y \,,\, F\,)}\,(\,T\,) \,=\, \left(\,\alpha\; \varphi_{(\,x \,,\, F\,)} \,+\, \beta\; \varphi_{(\,y \,,\, F\,)}\,\right)\,(\,T\,) \hspace{1.8cm}\]
\[=\, \left[\,\alpha\; J\,\left(\,x \,,\, b_{\,2} \,,\, \cdots \,,\, b_{\,n}\,\right) \,+\, \beta\; J\,\left(\,y \,,\, b_{\,2} \,,\, \cdots \,,\, b_{\,n}\,\right)\,\right]\,(\,T\,)\; \;\;\forall\; T \,\in\, X_{F}^{\,\ast}.\hspace{1.4cm}\]
\[\Rightarrow\; J\,\left(\,\alpha\,x \,+\, \beta\,y \,,\, b_{\,2} \,,\, \cdots \,,\, b_{\,n}\,\right) \,=\, \alpha\; J\,\left(\,x \,,\, b_{\,2} \,,\, \cdots \,,\, b_{\,n}\,\right) \,+\, \beta\; J\,\left(\,y \,,\, b_{\,2} \,,\, \cdots \,,\, b_{\,n}\,\right).\]
This shows that \,$J$ is a \,$b$-linear operator.
\item[(II)]\hspace{.1cm} $J$\; preserves the norm:\\
For each \,$\left(\,x \,,\, b_{\,2} \,,\, \cdots \,,\, b_{\,n}\,\right) \,\in\, X \,\times\, \left<\,b_{\,2}\,\right> \,\times \cdots \,\times\, \left<\,b_{\,n}\,\right>$, we have  
\[\left\|\,J\,\left(\,x \,,\, b_{\,2} \,,\, \cdots \,,\, b_{\,n}\,\right)\,\right\| \,=\, \left\|\,\varphi_{(\,x \,,\, F\,)}\,\right\| \,=\, \sup\left\{\,\dfrac{\left|\,\varphi_{(\,x \,,\, F\,)}\,(\,T\,)\,\right|}{\|\,T\,\|} \,:\, T \,\in\, X_{F}^{\,\ast} \,,\, T \,\neq\, 0\,\right\}\]
\[\hspace{3.5cm}=\, \sup\left\{\,\dfrac{\left|\,T\,(\,x \,,\, b_{\,2} \,,\, \cdots \,,\, b_{\,n}\,)\,\right|}{\|\,T\,\|} \,:\, T \,\in\, X_{F}^{\,\ast} \,,\, T \,\neq\, 0\,\right\}\]
\begin{equation}\label{eq2}
\hspace{2cm}=\, \left\|\,x \,,\, b_{\,2} \,,\, \cdots \,,\, b_{\,n}\,\right\|\; \;[\;\text{by Theorem (\ref{th1.3})}\;].
\end{equation}
\item[(III)]\hspace{.1cm} $J$\; is injective:\\
Let \,$x,\, y \,\in\, X$\; with \,$x \,\neq\, y$\, such that the set \,$\left\{\,x,\, b_{\,2},\, \cdots,\, b_{\,n}\,\right\}$\, or \,$\left\{\,y,\, b_{\,2},\, \cdots,\, b_{\,n}\,\right\}$\, are linearly independent.\;Then by (\ref{eq2}),
\[ \left\|\,x \,-\, y \,,\, b_{\,2} \,,\, \cdots \,,\, b_{\,n}\,\right\| \,\neq\, 0 \,\Rightarrow\, \left\|\,J\,\left(\,x \,-\, y \,,\, b_{\,2} \,,\, \cdots \,,\, b_{\,n}\,\right)\,\right\| \,\neq\, 0\]
\[\Rightarrow\, \left\|\,J\,\left(\,x \,,\, b_{\,2} \,,\, \cdots \,,\, b_{\,n}\,\right) \,-\, J\,\left(\,y \,,\, b_{\,2} \,,\, \cdots \,,\, b_{\,n}\,\right)\,\right\| \,\neq\, 0\]
\[\Rightarrow\, J\,\left(\,x \,,\, b_{\,2} \,,\, \cdots \,,\, b_{\,n}\,\right) \,\neq\, J\,\left(\,y \,,\, b_{\,2} \,,\, \cdots \,,\, b_{\,n}\,\right).\hspace{1cm}\]
We thus conclude that \,$J$\; is an isomeric isomorphism of \,$X \,\times\, \left<\,b_{\,2}\,\right> \,\times \cdots \,\times\, \left<\,b_{\,n}\,\right>$\; onto the subspace of \,$X_{F}^{\,\ast\,\ast}$.\;This completes the proof.     
\end{itemize}     
\end{proof}

\begin{definition}
Let \,$X$\; be a linear n-normed space over the field \,$\mathbb{R}$.\;The isometric isomorphism \,$J \,:\, X \,\times\, \left<\,b_{\,2}\,\right> \,\times \cdots \,\times\, \left<\,b_{\,n}\,\right> \,\to\, X_{F}^{\,\ast\,\ast}$\, defined by 
\[J\,\left(\,x \,,\, b_{\,2} \,,\, \cdots \,,\, b_{\,n}\,\right) \,=\, \varphi_{(\,x \,,\, F\,)}\; \;\forall\; x \,\in\, X\; \;\&\; \;\varphi_{(\,x \,,\, F\,)} \,\in\, X_{F}^{\,\ast\,\ast}\]
is called the b-natural embedding or the b-canonical mapping of \,$X \,\times\, \left<\,b_{\,2}\,\right> \,\times \cdots \,\times\, \left<\,b_{\,n}\,\right>$\; into the second dual space \,$X_{F}^{\,\ast\,\ast}$. 
\end{definition}

\begin{definition}
A linear n-normed space \,$X$\, is said to be b-reflexive if the b-natural embedding \,$J$, maps the space \,$X \,\times\, \left<\,b_{\,2}\,\right> \,\times \cdots \,\times\, \left<\,b_{\,n}\,\right>$\, onto its second dual space \,$X_{F}^{\,\ast\,\ast}$, i.\,e., \,$J\,\left(\,X \,\times\, \left<\,b_{\,2}\,\right> \,\times \cdots \,\times\, \left<\,b_{\,n}\,\right>\,\right) \,=\, X_{F}^{\,\ast\,\ast}$.
\end{definition}

\begin{theorem}
Let \,$\left\{\,x_{k}\,\right\}_{k \,=\, 1}^{\,\infty}$\, be a sequence in a linear n-normed space \,$X$.\;Suppose 
\[\sup\limits_{1 \,\leq\, k \,<\, \infty}\,\left|\,T \left(\,x_{k} \,,\, b_{\,2} \,,\, \cdots \,,\, b_{\,n}\,\right)\,\right| \,<\, \infty\; \;\;\forall\; T \,\in\, X_{F}^{\,\ast}.\;\text{Then}\] 
\[\sup\limits_{1 \,\leq\, k \,<\, \infty}\,\left\|\,x_{\,k} \,,\, b_{\,2} \,,\, \cdots \,,\, b_{\,n}\,\right\| \,<\, \infty.\]
\end{theorem}

\begin{proof}
Consider the \,$b$-natural embedding 
\[\left(\,x \,,\, b_{\,2} \,,\, \cdots \,,\, b_{\,n}\,\right) \,\to\, \varphi_{(\,x \,,\, F\,)},\; \left(\,x \,,\, b_{\,2} \,,\, \cdots \,,\, b_{\,n}\,\right) \,\in\, X \,\times\, \left<\,b_{\,2}\,\right> \,\times \cdots \,\times\, \left<\,b_{\,n}\,\right>.\]Since \,$\left\{\,x_{k}\,\right\}_{k \,=\, 1}^{\,\infty}$\; is a sequence of vectors in \,$X$, \,$\left\{\,\varphi_{(\,x_{k} \,,\, F\,)}\,\right\}_{k \,=\, 1}^{\,\infty}$\; is a sequence of bounded linear functionals in \,$X_{F}^{\,\ast\,\ast}$.\;Also,
\[\left|\,\varphi_{(\,x_{k} \,,\, F\,)}\,(\,T\,)\,\right| \,=\, \left|\,T \left(\,x_{k} \,,\, b_{\,2} \,,\, \cdots \,,\, b_{\,n}\,\right)\,\right| \,\leq\, \sup\limits_{1 \,\leq\, k \,<\, \infty}\left|\,T \left(\,x_{k} \,,\, b_{\,2} \,,\, \cdots \,,\, b_{\,n}\,\right)\,\right|.\]
Therefore, \,$\left\{\,\varphi_{(\,x_{k} \,,\, F\,)}\,(\,T\,)\,\right\}_{k \,=\, 1}^{\,\infty}$\; is bounded for each \,$T \,\in\, X_{F}^{\,\ast}$.\;Applying the Principle of Uniform Boundedness (\,Theorem (\ref{th0.01})\,), to the family \,$\left\{\,\varphi_{(\,x_{k} \,,\, F\,)}\,\right\}_{k \,=\, 1}^{\,\infty}$, we conclude that \,$\left\{\,\left\|\,\varphi_{(\,x_{k} \,,\, F\,)}\,\right\|\,\right\}_{k \,=\, 1}^{\,\infty}$\; is bounded and hence by (\ref{eq2}), \,$\left\{\,\left\|\,x_{k} \,,\, b_{\,2} \,,\, \cdots \,,\, b_{\,n}\,\right\|\,\right\}_{k \,=\, 1}^{\,\infty}$\; is bounded.\;This proves the theorem.      
\end{proof}

\begin{theorem}
A closed subspace of a b-reflexive n-Banach space is b-reflexive.
\end{theorem}

\begin{proof}
Let \,$X$\; be a \,$b$-reflexive \,$n$-Banach space and \,$Y$\, be a closed subspace of \,$X$.\;Let \,$T \,:\, X_{F}^{\,\ast} \,\to\, Y_{F}^{\,\ast}$\, be an operator defined by 
\[\left(\,T\,f\,\right) \left(\,y \,,\, b_{\,2} \,,\, \cdots \,,\, b_{\,n}\,\right) \,=\, f\,\left(\,y \,,\, b_{\,2} \,,\, \cdots \,,\, b_{\,n}\,\right)\; \;\forall\; y \,\in\, Y,\; f \,\in\, X_{F}^{\,\ast},\]where \,$Y_{F}^{\,\ast}$\, denotes the Banach space of all bounded \,$b$-linear functionals defined on \,$Y \,\times\, \left<\,b_{\,2}\,\right> \,\times \cdots \,\times\, \left<\,b_{\,n}\,\right>$.\;Then for \,$f \,\in\, X_{F}^{\,\ast}$,
\[\left\|\,T\,f\,\right\| \,=\, \sup\left\{\,\dfrac{\left|\,f\,(\,y \,,\, b_{\,2} \,,\, \cdots \,,\, b_{\,n}\,)\,\right|}{\left\|\,y \,,\, b_{\,2} \,,\, \cdots \,,\, b_{\,n}\,\right\|} \,:\, \left\|\,y \,,\, b_{\,2} \,,\, \cdots \,,\, b_{\,n}\,\right\| \,\neq\, 0\,\right\} \,=\, \|\,f\,\|.\]
Let \,$J_{Y}$\, be the \,$b$-natural embedding of \,$Y \,\times\, \left<\,b_{\,2}\,\right> \,\times \cdots \,\times\, \left<\,b_{\,n}\,\right>$\, into \,$Y_{F}^{\,\ast\,\ast}$.\;That is,
\[J_{Y}\,\left(\,y \,,\, b_{\,2} \,,\, \cdots \,,\, b_{\,n}\,\right) \,=\, \psi_{(\,y \,,\, F\,)}\; \;\forall\; y \,\in\, Y, \;\psi_{(\,y \,,\, F\,)} \,\in\, Y_{F}^{\,\ast\,\ast}.\;\text{Define}\]\,$T_{\,1} \,:\, Y_{F}^{\,\ast\,\ast} \,\to\, X_{F}^{\,\ast\,\ast}$\; by \,$\left(\,T_{\,1}\,\psi_{(\,y \,,\, F\,)}\,\right)\,(\,f\,) \,=\, \psi_{(\,y \,,\, F\,)}\,(\,T\,f\,),\; f \,\in\, X_{F}^{\,\ast}$.\;We now verify that \,$T_{\,1}\,\psi_{(\,y \,,\, F\,)} \,\in\, X_{F}^{\,\ast\,\ast}$.
\begin{itemize}
\item[(I)]\hspace{.1cm} \,$T_{\,1}\,\psi_{(\,y \,,\, F\,)}$\; is linear functional:\\
Let \,$\alpha,\, \beta \,\in\, \mathbb{R}$.\;Then for every \,$f,\, g \,\in\, X^{\,\ast}_{F}$\, and \,$y \,\in\, Y$, we have
\[\left(\,T_{\,1}\,\psi_{(\,y \,,\, F\,)}\,\right)\,\left(\,\alpha\,f \,+\, \beta\,g\,\right)\,\left(\,y \,,\, b_{\,2} \,,\, \cdots \,,\, b_{\,n}\,\right)\hspace{6cm}\]
\[\,=\, \psi_{(\,y \,,\, F\,)}\,\left[\,T\,\left(\,\alpha\,f \,+\, \beta\,g\,\right)\,\right]\,\left(\,y \,,\, b_{\,2} \,,\, \cdots \,,\, b_{\,n}\,\right)\hspace{4.5cm}\]
\[\,=\, \psi_{(\,y \,,\, F\,)}\,\left[\,\alpha\; T\,\left(\,f\,\left(\,y \,,\, b_{\,2} \,,\, \cdots \,,\, b_{\,n}\,\right)\,\right) \,+\, \beta\; T\,\left(\,g\,\left(\,y \,,\, b_{\,2} \,,\, \cdots \,,\, b_{\,n}\,\right)\,\right) \,\right] \hspace{.7cm}\]
\[\,=\, \alpha\; \psi_{(\,y \,,\, F\,)}\,(\,T\,f\,)\,\left(\,y \,,\, b_{\,2} \,,\, \cdots \,,\, b_{\,n}\,\right) \,+\, \beta\; \psi_{(\,y \,,\, F\,)}\,(\,T\,g\,)\,\left(\,y \,,\, b_{\,2} \,,\, \cdots \,,\, b_{\,n}\,\right)\]
\[=\, \left[\,\alpha\; \psi_{(\,y \,,\, F\,)}\,(\,T\,f\,) \,+\,  \beta\; \psi_{(\,y \,,\, F\,)}\,(\,T\,g\,)\,\,\right]\,\left(\,y \,,\, b_{\,2} \,,\, \cdots \,,\, b_{\,n}\,\right)\hspace{2.4cm}\] 
\[=\, \left[\,\alpha\; \left(\,T_{\,1}\,\psi_{(\,y \,,\, F\,)}\,\right)\,(\,f\,) \,+\, \beta\; \left(\,T_{\,1}\,\psi_{(\,y \,,\, F\,)}\,\right)\,(\,g\,)\,\right]\,\left(\,y \,,\, b_{\,2} \,,\, \cdots \,,\, b_{\,n}\,\right).\hspace{.6cm}\]
\[\Rightarrow\, \left(\,T_{\,1}\,\psi_{(\,y \,,\, F\,)}\,\right)\,\left(\,\alpha\,f \,+\, \beta\,g\,\right) \,=\, \alpha\; \left(\,T_{\,1}\,\psi_{(\,y \,,\, F\,)}\,\right)\,(\,f\,) \,+\, \beta\; \left(\,T_{\,1}\,\psi_{(\,y \,,\, F\,)}\,\right)\,(\,g\,).\]  
\item[(II)]\hspace{.1cm} $T_{\,1}\, \psi_{(\,y \,,\, F\,)}$\; is bounded:\\
Since \,$\psi_{(\,y \,,\, F\,)}$\, preserves the norm,
\[\left\|\,\left(\,T_{\,1}\,\psi_{(\,y \,,\, F\,)}\,\right)\,(\,f\,)\,\right\| \,=\, \left\|\,\psi_{(\,y \,,\, F\,)}\,(\,T\,f\,)\,\right\| \,=\, \|\,T\,f\,\| \,=\, \|\,f\,\|.\]  
\end{itemize}
So, \,$T_{\,1}\, \psi_{(\,y \,,\, F\,)} \,\in\, X_{F}^{\,\ast\,\ast}$\, and hence \,$T_{\,1}$\, is well-defined.\;Since \,$X$\, is \,$b$-reflexive, the \,$b$-natural embedding \,$J_{X} \,:\, X \,\times\, \left<\,b_{\,2}\,\right> \,\times \cdots \,\times\, \left<\,b_{\,n}\,\right> \,\to\, X_{F}^{\,\ast\,\ast}$\, defined by 
\[J_{X}\,\left(\,x \,,\, b_{\,2} \,,\, \cdots \,,\, b_{\,n}\,\right) \,=\, \varphi_{(\,x \,,\, F\,)}\;, \;\varphi_{(\,x \,,\, F\,)} \,\in\, X_{F}^{\,\ast\,\ast}\] 
is such that \,$J_{X}\,\left(\,X \,\times\, \left<\,b_{\,2}\,\right> \,\times \cdots \,\times\, \left<\,b_{\,n}\,\right>\,\right) \,=\, X_{F}^{\,\ast\,\ast}$.\;Therefore, \,$T_{\,1} \psi_{(\,y \,,\, F\,)} \,\in\, X_{F}^{\,\ast\,\ast}$\, implies that \,$J_{X}^{\,-\, 1}\,\left(\,T_{\,1}\, \psi_{(\,y \,,\, F\,)}\,\right) \,\in\, X \,\times\, \left<\,b_{\,2}\,\right> \,\times \cdots \,\times\, \left<\,b_{\,n}\,\right>$.\;Write \,$\left(\,x \,,\, b_{\,2} \,,\, \cdots \,,\, b_{\,n}\,\right) \,=\, J_{X}^{\,-\, 1}\,\left(\,T_{\,1}\, \psi_{(\,y \,,\, F\,)}\,\right)$\, so that \,$J_{X}\,\left(\,x \,,\, b_{\,2} \,,\, \cdots \,,\, b_{\,n}\,\right) \,=\, T_{\,1}\,\psi_{(\,y \,,\, F\,)}$.\;We need to prove that \,$x \,\in\, Y$.\;Let, if possible, \,$x \,\in\, X \,-\, Y$\, such that \,$x,\, b_{\,2},\, \cdots,\, b_{\,n}$\, are linearly independent.\;Then by Corollary (\ref{cor4}), \,$\exists$\; a bounded \,$b$-linear functional \,$f \,\in\, X_{F}^{\,\ast}$\, such that \,$f\,\left(\,x \,,\, b_{\,2} \,,\, \cdots \,,\, b_{\,n}\,\right) \,\neq\, 0$\, and \,$f\,\left(\,y \,,\, b_{\,2} \,,\, \cdots \,,\, b_{\,n}\,\right) \,=\, 0$\, for all \,$y \,\in\, Y$.\;Consequently, \,$T\,f \,=\, 0$\; and as such \,$\psi_{(\,y \,,\, F\,)}\,(\,T\,f\,) \,=\, 0$.\;This leads to \,$\varphi_{(\,x \,,\, F\,)}\,(\,f\,)\\ \,=\, 0$\, and hence \,$f\,\left(\,x \,,\, b_{\,2} \,,\, \cdots \,,\, b_{\,n}\,\right) \,=\, 0$, which is a contradiction.\;Thus, we conclude that \,$\left(\,x \,,\, b_{\,2} \,,\, \cdots \,,\, b_{\,n}\,\right) \,=\, J_{X}^{\,-\, 1}\,\left(\,T_{\,1}\,\psi_{(\,y \,\, F\,)}\,\right) \,\in\, Y \,\times\, \left<\,b_{\,2}\,\right> \,\times \cdots \,\times\, \left<\,b_{\,n}\,\right>$.\;This verifies that \,$J_{X}^{\,-\, 1}\,\left(\,T_{\,1}\,\left(\, Y_{F}^{\,\ast\,\ast}\,\right) \,\right) \,\subset\, Y \,\times\, \left<\,b_{\,2}\,\right> \,\times \cdots \,\times\, \left<\,b_{\,n}\,\right>$.\;Now, let \,$\psi \,\in\, Y_{F}^{\,\ast\,\ast}$.\;Set \,$\left(\,x_{\,0} \,,\, b_{\,2} \,,\, \cdots \,,\, b_{\,n}\,\right) \,=\, J_{X}^{\,-\, 1}\,\left(\,T_{\,1}\,\psi\,\right)$\, so that \,$\left(\,x_{\,0} \,,\, b_{\,2} \,,\, \cdots \,,\, b_{\,n}\,\right) \,\in\, Y \,\times\, \left<\,b_{\,2}\,\right> \,\times \cdots \,\times\, \left<\,b_{\,n}\,\right>$.\;Let \,$g \,\in\, Y_{F}^{\,\ast}$.\;Then there exists a \,$b$-linear functional \,$f \,\in\, X_{F}^{\,\ast}$\, such that 
\[f\,\left(\,y \,,\, b_{\,2} \,,\, \cdots \,,\, b_{\,n}\,\right) \,=\, g\,\left(\,y \,,\, b_{\,2} \,,\, \cdots \,,\, b_{\,n}\,\right)\; \;\forall\; y \,\in\, Y\; \;\text{and}\; \;g \,=\, T\,f.\]
\[\text{Therefore,}\; \;\psi\,(\,g\,) \,=\, \psi\,(\,T\,f\,) \,=\, \left(\,T_{\,1}\,\psi\,\right)\,(\,f\,) \,=\, \left[\,J_{X}\,\left(\,x_{\,0} \,,\, b_{\,2} \,,\, \cdots \,,\, b_{\,n}\,\right)\,\right]\,(\,f\,)\hspace{1cm}\]
\[\hspace{2.5cm}=\, \varphi_{(\,x_{\,0} \,,\, F\,)}\,(\,f\,) \,=\, f\,\left(\,x_{\,0},\, b_{\,2},\, \cdots,\, b_{\,n}\,\right) \,=\, g\,\left(\,x_{\,0},\, b_{\,2},\, \cdots,\, b_{\,n}\,\right).\]This proves that \,$J_{Y}\,\left(\,x_{\,0} \,,\, b_{\,2} \,,\, \cdots \,,\, b_{\,n}\,\right) \,=\, \psi_{(\,x_{\,0} \,,\, F\,)}$\, and hence 
\[J_{Y}\,\left(\,Y \,\times\, \left<\,b_{\,2}\,\right> \,\times \cdots \,\times\, \left<\,b_{\,n}\,\right>\,\right) \,=\, Y_{F}^{\,\ast\,\ast}.\;\text{This proves that \,$Y$\, is $b$-reflexive.}\]   
\end{proof}

\section{$b$-weak convergence and $b$-strong convergence in linear $n$-normed space}

\smallskip\hspace{.6 cm}In this section, we shall introduce \,$b$-weak convergence and \,$b$-strong convergence relative to bounded\;$b$-linear functionals in linear\;$n$-normed space and establish that these two types of convergence are equivalent in case of finite dimensional linear\;$n$-normed space.

\begin{definition}
A sequence \,$\{\,x_{\,k}\,\}$\; in a linear n-normed space \,$X$ is said to be b-weakly convergent if \,$\exists$\, an element \,$x \,\in\, X$\; such that for every \,$T \,\in\, X^{\,\ast}_{F}$,
\[\lim\limits_{k \,\to\, \infty}\,T\,(\, x_{\,k} \,,\, b_{\,2} \,,\, \cdots \,,\, b_{\,n}\,) \,=\, T\,(\,x \,,\, b_{\,2} \,,\, \cdots \,,\, b_{\,n}\,).\]
The vector \,$x$\, is called the b-weak limit of the sequence \,$\{\,x_{\,k}\,\}$\, and we say that \,$\{\,x_{\,k}\,\}$\, converges b-weakly to \,$x$.\;Note that, for each \,$T \,\in\, X^{\,\ast}_{F},\; \left\{\,T\,(\, x_{\,k} \,,\, b_{\,2} \,,\, \cdots \,,\, b_{\,n}\,)\,\right\}$\; is a sequence of scalars in \,$\mathbb{K}$.\;Therefore, b-weak convergence means convergence of the sequence of scalars \,$\left\{\,T\,(\, x_{\,k} \,,\, b_{\,2} \,,\, \cdots \,,\, b_{\,n}\,)\,\right\}$\, for every \,$T \,\in\, X^{\,\ast}_{F}$.  
\end{definition}

\begin{theorem}\label{th4.1}
Let \,$\{\,x_{\,k}\,\}$\, be b-weakly convergent sequence in \,$X$.\;Then
\begin{itemize}
\item[(I)]\hspace{.1cm} the b-weak limit of \,$\{\,x_{\,k}\,\}$\; is unique.
\item[(II)]\hspace{.1cm}\,$\left\{\,\left\|\,x_{\,k} \,,\, b_{\,2} \,,\, \cdots \,,\, b_{\,n}\,\right\|\,\right\}$\, is bounded sequence in \,$\mathbb{K}$.
\end{itemize}
\end{theorem}

\begin{proof}(I)\hspace{.1cm}
Suppose that \,$\{\,x_{\,k}\,\}$\; converges \,$b$-weakly to \,$x$\; as well as to \,$y$.\;Then 
\[T\,(\,x \,,\, b_{\,2} \,,\, \cdots \,,\, b_{\,n}\,) \,=\, \lim\limits_{k \;\to\; \infty}\,T\,(\, x_{\,k} \,,\, b_{\,2} \,,\, \cdots \,,\, b_{\,n}\,) \,=\, T\,(\,y \,,\, b_{\,2} \,,\, \cdots \,,\, b_{\,n}\,)\; \;\forall\; T \,\in\, X^{\,\ast}_{F}\]
\[\Rightarrow\, T\,(\,x \,,\, b_{\,2} \,,\, \cdots \,,\, b_{\,n}\,) \,-\, T\,(\,y \,,\, b_{\,2} \,,\, \cdots \,,\, b_{\,n}\,) \,=\, 0\; \;\;\forall\; \,T \,\in\, X^{\,\ast}_{F}\]
\[\Rightarrow\; T\,(\, x \,-\, y \,,\, b_{\,2} \,,\, \cdots \,,\, b_{\,n}\,) \,=\, 0\; \;\;\forall\; \,T \,\in\, X^{\,\ast}_{F}.\hspace{2.7cm}\]
Hence, by Corollary (\ref{cor2}), \,$x \,=\, y$.\\\\
Proof of (II)\hspace{.1cm}
Since \,$\{\,x_{\,k}\,\}$\; converges \,$b$-weakly to \,$x$, we have
\[\lim\limits_{k \;\to\; \infty}\,T\,(\,x_{\,k} \,,\, b_{\,2} \,,\, \cdots \,,\, b_{\,n}\,) \,=\, T\,(\, x \,,\, b_{\,2} \,,\, \cdots \,,\, b_{\,n}\,)\; \;\;\forall\; \,T \,\in\, X^{\,\ast}_{F}.\]
Therefore, for each \,$T \,\in\, X^{\,\ast}_{F} \;,\; \left\{\,T \left(\,x_{\,k} \,,\, b_{\,2} \,,\, \cdots \,,\, b_{\,n}\,\right)\,\right\}$\, is a convergent sequence in \,$\mathbb{K}$\, and so the sequence \,$\left\{\,T\,(\,x_{\,k} \,,\, b_{\,2} \,,\, \cdots \,,\, b_{\,n}\,)\,\right\}$\, is bounded.\;Consequently, \,$\exists$\, a constant \,$K_{T}$\, (\,depending on \,$T$\,) such that \,$\left|\,T\,(\,x_{\,k} \,,\, b_{\,2} \,,\, \cdots \,,\, b_{\,n}\,)\,\right| \,\leq\, K_{T}\, \;\forall\; k \,\in\, \mathbb{N}$.
Let \,$\left(\,x \,,\, b_{\,2} \,,\, \cdots \,,\, b_{\,n}\,\right) \,\to\, \varphi_{(\,x \,,\, F\,)}$\, be the \,$b$-natural embedding of \,$X \,\times\, \left<\,b_{\,2}\,\right> \,\times \cdots \,\times\, \left<\,b_{\,n}\,\right>$\, into \,$X^{\,\ast\,\ast}_{F}$.\;Then for each \,$k \,\in\, \mathbb{N}\,, \;\left\|\,\varphi_{(\,x_{\,k} \,,\, F\,)}\,\right\| \,=\, \left\|\,x_{\,k} \,,\, b_{\,2} \,,\, \cdots \,,\, b_{\,n}\,\right\|\; \;[\;\text{by (\ref{eq2})}\;]$, and
\[\left|\,\varphi_{(\,x_{k} \,,\, F\,)}\,(\,T\,)\,\right| \,=\, \left|\,T\,(\,x_{\,k} \,,\, b_{\,2} \,,\, \cdots \,,\, b_{\,n}\,)\,\right| \,\leq\, K_{T}\; \;\;\forall\; k \,\in\, \mathbb{N}.\]
Thus, \,$\left\{\,\varphi_{(\,x_{k} \,,\, F\,)}\,(\,T\,)\,\right\}$\; is bounded for each \,$T \,\in\, X^{\,\ast}_{F}$.\;But the space \,$X^{\,\ast}_{F}$\, being a Banach space, by the Principle of Uniform Boundedness (\,Theorem (\ref{th0.01})\,), it follows that \,$\left\{\,\left\|\,\varphi_{(\,x_{\,k} \,,\, F\,)}\,\right\|\,\right\}$\; is bounded and hence $\left\{\,\left\|\,x_{\,k} \,,\, b_{\,2} \,,\, \cdots \,,\, b_{\,n}\,\right\|\,\right\}_{k \,=\, 1}^{\,\infty}$\; is bounded.     
\end{proof}

\begin{theorem}
Let \,$\{\,x_{\,k}\,\}$\; and \,$\{\,y_{\,k}\,\}$\; be two sequences in a linear n-normed space \,$X$.\;If \,$\{\,x_{\,k}\,\}$\; and \,$\{\,y_{\,k}\,\}$\; converges b-weakly to \,$x$\, and \,$y$, respectively then for any scalar \,$\alpha\, \;\text{and}\; \,\beta$, \,$\{\,\alpha\, x_{\,k} \,+\, \beta\, y_{\,k}\,\}$\; converges b-weakly to \,$\alpha\,x \,+\, \beta\,y$.
\end{theorem}

\begin{proof}
Since \,$\{\,x_{\,k}\,\}$\; and \,$\{\,y_{\,k}\,\}$\; converges \,$b$-weakly to \,$x$\; and \,$y$, 
\[\lim\limits_{k \;\to\; \infty}\,T\,(\,x_{\,k} \,,\, b_{\,2} \,,\, \cdots \,,\, b_{\,n}\,) \,=\, T\,(\, x \,,\, b_{\,2} \,,\, \cdots \,,\, b_{\,n}\,)\; \;\text{and}\] 
\[\lim\limits_{k \;\to\; \infty}\,T\,(\,y_{\,k} \,,\, b_{\,2} \,,\, \cdots \,,\, b_{\,n}\,) \,=\, T\,(\,y \,,\, b_{\,2} \,,\, \cdots \,,\, b_{\,n}\,)\; \;\;\forall\; T \,\in\, X^{\,\ast}_{F}.\]
Now, for all \,$T \,\in\, X^{\,\ast}_{F}$,  \,$\lim\limits_{k \;\to\; \infty}\,T\,(\,\alpha\,x_{\,k} \,+\, \beta\,y_{\,k} \,,\, b_{\,2} \,,\, \cdots \,,\, b_{\,n}\,)$
\[\,=\, \lim\limits_{k \;\to\; \infty}\,\left[\,T\,(\,\alpha\, x_{\,k} \,,\, b_{\,2} \,,\, \cdots \,,\, b_{\,n}\,) \,+\, T\,(\,\beta\, y_{\,k} \,,\, b_{\,2} \,,\, \cdots \,,\, b_{\,n}\,)\,\right]\hspace{3.6cm}\]
\[\,=\, \lim\limits_{k \;\to\; \infty}\,\alpha\,T\,(\,x_{\,k} \,,\, b_{\,2} \,,\, \cdots \,,\, b_{\,n}\,) \,+\, \lim\limits_{k \;\to\; \infty}\,\beta\,T\,(\,y_{\,k} \,,\, b_{\,2} \,,\, \cdots \,,\, b_{\,n}\,)\hspace{3cm}\]
\[ =\; \alpha\,T\,(\,x \,,\, b_{\,2} \,,\, \cdots \,,\, b_{\,n}\,) \,+\, \beta\,T\,(\,y \,,\, b_{\,2} \,,\, \cdots \,,\, b_{\,n}\,) \,=\, T\,(\,\alpha\,x \,+\, \beta\,y \,,\, b_{\,2} \,,\, \cdots \,,\, b_{\,n}\,).\]
This shows that \,$\{\,\alpha\,x_{\,k} \,+\, \beta\,y_{\,k}\,\}$\; converges \,$b$-weakly to \,$\alpha\,x \,+\, \beta\,y$.
\end{proof}

\begin{theorem}
A sequence \,$\{\,x_{\,k}\,\}$\; in \,$X$\, converges b-weakly to \,$x \,\in\, X$\; if and only if
\begin{itemize}
\item[(I)]\hspace{.2cm} the sequence \,$\left\{\,\left\|\,x_{\,k} \,,\, b_{\,2} \,,\, \cdots \,,\, b_{\,n}\,\right\|\,\right\}$\; is bounded and
\item[(II)]\hspace{.2cm} $\lim\limits_{k \;\to\; \infty}\,T\,(\,x_{\,k} \,,\, b_{\,2} \,,\, \cdots \,,\, b_{\,n}\,) \,=\, T\,(\,x \,,\, b_{\,2} \,,\, \cdots \,,\, b_{\,n}\,)\; \;\forall\; T \,\in\, M$, where \,$M$\, is fundamental or total subset of \,$X^{\,\ast}_{F}$.
\end{itemize} 
\end{theorem}

\begin{proof}
In the case of \,$b$-weak convergence, (I) follows from the Theorem (\ref{th4.1}) and since \,$M \,\subset\, X^{\,\ast}_{F}$, (II) follows from the definition of \,$b$-weak convergence of \,$\{\,x_{\,k}\,\}$.\\\\
Conversely, suppose that (I) and (II) hold\,.\;By (I), \,$\exists$\; a constant \,$L$\; such that
\[\left\|\,x_{\,k} \,,\, b_{\,2} \,,\, \cdots \,,\, b_{\,n}\,\right\|\, \leq\, L\; \;\forall\; k \,\in\, \mathbb{N}\; \;\;\text{and}\; \;\left\|\,x \,,\, b_{\,2} \,,\, \cdots \,,\, b_{\,n}\,\right\| \,\leq\, L.\]
Since \,$\overline{span\,M} \,=\, X^{\,\ast}_{F}$, for each \,$T \,\in\, X^{\,\ast}_{F}$, \,$\exists$\; a sequence \,$\left\{\,T_{\,m}\,\right\}$\; in \,$span\,M$\; such that \,$\lim\limits_{m \;\to\; \infty}\, T_{\,m} \,=\, T$.\;Hence, for any given \,$\epsilon \,>\, 0,\, \;\exists\; \,T_{\,m} \,\in\, span\,M$\; such that \,$\left\|\,T_{\,m} \,-\, T \,\right\| \,<\, \dfrac{\epsilon}{3\,L}$.\;Furthermore, by the hypothesis (II), \,$\exists\; K \,\in\, \mathbb{N}$\, such that
\[\left|\,T_{\,m}\,(\,x_{\,k} \,,\, b_{\,2} \,,\, \cdots \,,\, b_{\,n}\,) \,-\, T_{\,m}\,(\,x \,,\, b_{\,2} \,,\, \cdots \,,\, b_{\,n}\,)\,\right| \,<\, \dfrac{\epsilon}{3}\; \,\;\forall\; \,m \,>\, K.\]
Now, for \,$m \,>\, K$, \,$\left|\,T\,(\,x_{\,k} \,,\, b_{\,2} \,,\, \cdots \,,\, b_{\,n}\,) \,-\, T\,(\,x \,,\, b_{\,2} \,,\, \cdots \,,\, b_{\,n}\,)\,\right|$
\[\leq\, \left|\,T\,(\,x_{\,k} \,,\, b_{\,2} \,,\, \cdots \,,\, b_{\,n}\,) \,-\, T_{\,m}\,(\,x_{\,k} \,,\, b_{\,2} \,,\, \cdots \,,\, b_{\,n}\,)\,\right| \,+\,\]
\[\hspace{1.3cm}\,+\, \left|\,T_{\,m}\,(\,x_{\,k} \,,\, b_{\,2} \,,\, \cdots \,,\, b_{\,n}\,) \,-\, T_{\,m}\,(\,x \,,\, b_{\,2} \,,\, \cdots \,,\, b_{\,n}\,)\,\right| \]
\[\hspace{1.7cm}+\; \left|\,T_{\,m}\,(\;x \;,\; b_{\,2} \,,\, \cdots \,,\, b_{\,n}\;) \;-\;T\,(\;x \;,\; b_{\,2} \,,\, \cdots \,,\, b_{\,n}\;)\,\right|  \]
\[ <\, \left\|\,T_{\,m} \,-\, T\,\right\| \left\|\,x_{\,k} \,,\, b_{\,2} \,,\, \cdots \,,\, b_{\,n}\,\right\| \,+\, \dfrac{\epsilon}{3} \,+\, \left\|\,T_{\,m} \,-\, T\,\right\| \left\|\,x \,,\, b_{\,2} \,,\, \cdots \,,\, b_{\,n}\,\right\|\]
\[<\, \dfrac{\epsilon}{3\,L}\, \cdot\,L \,+\, \dfrac{\epsilon}{3} \,+\, \dfrac{\epsilon}{3\,L}\,\cdot\,L \,=\, \dfrac{\epsilon}{3} \,+\,  \dfrac{\epsilon}{3} \,+\, \dfrac{\epsilon}{3} \,=\, \epsilon \hspace{4cm}\]
\[\Rightarrow\, \lim\limits_{k \;\to\; \infty}\,T\,(\,x_{\,k} \,,\, b_{\,2} \,,\, \cdots \,,\, b_{\,n}\,) \,=\, T\,(\,x \,,\, b_{\,2} \,,\, \cdots \,,\, b_{\,n}\,)\; \;\;\forall\; T \,\in\, X^{\,\ast}_{F}.\hspace{1cm}\]
Hence, \,$\{\,x_{\,k}\,\}$\; converges \,$b$-weakly to \,$x \,\in\, X$.    
\end{proof}

\begin{definition}
A sequence \,$\{\,x_{\,k}\,\}$\, in \,$X$ is said to be b-strongly convergent if \,$\exists$\; a vector \,$x \,\in\, X$\; such that \,$\lim\limits_{k \to \infty}\,\left\|\,x_{\,k} \,-\, x \,,\, b_{\,2} \,,\, \cdots \,,\, b_{\,n} \,\right\| \,=\, 0$.\;The vector \,$x$\; is called b-strong limit and we say that \,$\{\,x_{\,k}\,\}$\; converges b-strongly to \,$x$.
\end{definition}

\begin{theorem}
If a sequence \,$\{\,x_{\,k}\,\}$\, in \,$X$\, converges b-strongly to \,$x$, then \,$\{\,x_{\,k}\,\}$\, converges b-weakly to \,$x$\; in \,$X$.
\end{theorem}

\begin{proof}
Suppose \,$\{\,x_{\,k}\,\}$\, converges \,$b$-strongly to \,$x$.\;Then for every \,$T \,\in\, X^{\,\ast}_{F}$, we have
\[\left|\,T\,(\,x_{\,k} \,,\, b_{\,2} \,,\, \cdots \,,\, b_{\,n}\,) \,-\, T\,(\,x \,,\, b_{\,2} \,,\, \cdots \,,\, b_{\,n}\,)\,\right| \,=\, \left|\,T\,(\,x_{\,k} \,-\, x \,,\, b_{\,2} \,,\, \cdots \,,\, b_{\,n}\,)\,\right|\]
\[\leq\, \|\,T\,\| \left\|\,x_{\,k} \,-\, x \,,\, b_{\,2} \,,\, \cdots \,,\, b_{\,n}\,\right\|\hspace{.8cm}\]
\[\hspace{3.5cm} \,\to\, 0\; \;\text{as}\; \,k \,\to\, \infty\; \;[\;\text{since $\{\,x_{\,k}\,\}$\, converges \,$b$-strongly to \,$x$}\;]\] 
\[\Rightarrow\, \lim\limits_{k \;\to\; \infty}\,T\,(\, x_{\,k} \,,\, b_{\,2} \,,\, \cdots \,,\, b_{\,n}\,) \,=\, T\,(\,x \,,\, b_{\,2} \,,\, \cdots \,,\, b_{\,n}\,)\; \;\forall\; T \,\in\, X^{\,\ast}_{F}.\]
Hence, \,$\{\,x_{\,k}\,\}$\, converges \,$b$-weakly to \,$x$\, in \,$X$.
\end{proof}

\begin{theorem}
In a finite dimensional linear n-normed space, b-weak convergence implies b-strong convergence.
\end{theorem}

\begin{proof}
Let \,$X$\, be a linear\;$n$-normed space with \,$\text{dim}\,X \,=\, d \,\geq\, n$.\;Then, \,$\exists$\, a basis \,$\left\{\,e_{\,1} \,,\, e_{\,2} \,,\, \cdots \,,\, e_{\,d}\,\right\}$\; for \,$X$.\;Let \,$\{\,x_{\,k}\,\}$\; be a sequence in \,$X$\; such that \,$\{\,x_{\,k}\,\}$\; converges \,$b$-weakly to \,$x$.\;Now, we can write 
\[ x_{\,k} \,=\, a_{\,k \,,\, 1}\,e_{\,1} \,+\, a_{\,k \,,\, 2}\,e_{\,2} \,+\, \,\cdots \,+\, a_{\,k \,,\, d}\,e_{\,d} \;,\; (\,k \,=\, 1 \,,\, 2 \,,\, \cdots\,)\, \;\text{and}\] 
\[x \,=\, a_{\,1}\,e_{\,1} \,+\, a_{\,2}\,e_{\,2} \,+\,  \cdots \,\cdots \,+\, a_{\,d}\,e_{\,d},\]
where \,$a_{\,k \,,\, 1},\, a_{\,k \,,\, 2},\, \,\cdots \,,\, a_{\,k \,,\, d} \,,\, a_{\,1},\, a_{\,2},\, \,\cdots \,,\,  a_{\,d} \,\in\, \mathbb{R}$.\;Consider the \,$b$-linear functionals \,$\left\{\,T_{\,1} \,,\, T_{\,2} \,,\, \cdots \,,\, T_{\,d}\,\right\}$\; in \,$X^{\,\ast}_{F}$\, such that
\[ T_{\,i}\,(\,e_{\,j} \,,\, b_{\,2} \,,\, \cdots \,,\, b_{\,n}\,) \,=\, \begin{cases}
1 & \text{if\;\;}\; i \;=\; j \\ 0 & \text{if\;\;}\; i \;\neq\; j\,,\; 1\; \leq\; i,\, j \;\leq\; d \end{cases}\]
Now, for \,$1 \,\leq\, i \,\leq\, d$, we have
\[T_{\,i}\,(\,x_{\,k} \,,\, b_{\,2} \,,\, \cdots \,,\, b_{\,n}\,) \,=\, T_{\,i}\,\left(\,\sum\limits_{j \,=\, 1}^{\,d}\;a_{\,k \,,\, j}\,e_{\,j} \;,\; b_{\,2} \,,\, \cdots \,,\, b_{\,n}\,\right)\]
\[=\, \sum\limits_{j \,=\, 1}^{\,d}\;a_{\,k \,,\, j}\,T_{\,i}\,(\,e_{\,j} \,,\, b_{\,2} \,,\, \cdots \,,\, b_{\,n}\,) \,=\, a_{\,k \,,\, i}\] and similarly, \,$T_{\,i}\,(\,x \,,\, b_{\,2} \,,\, \cdots \,,\, b_{\,n}\,) \,=\, a_{\,i},\,(\,1 \,\leq\, i \,\leq\, d\,)$.\;Since
\[\lim\limits_{k \;\to\; \infty}\,T\,(\,x_{\,k} \,,\, b_{\,2} \,,\, \cdots \,,\, b_{\,n}\,) \,=\, T\,(\,x \,,\, b_{\,2} \,,\, \cdots \,,\, b_{\,n}\,)\; \;\;\forall\; T \,\in\, X^{\,\ast}_{F},\] in particular, we have
\[\lim\limits_{k \;\to\; \infty}\,T_{\,i}\,(\,x_{\,k} \,,\, b_{\,2} \,,\, \cdots \,,\, b_{\,n}\,) \,=\, T_{\,i}\,(\,x \,,\, b_{\,2} \,,\, \cdots \,,\, b_{\,n}\,),\; (\,1 \,\leq\, i \,\leq\, d\,)\]
\begin{equation}\label{eq3}
\Rightarrow\; \lim\limits_{k \;\to\; \infty}\, a_{k \,,\, i} \,=\, a_{\,i}\,,\; (\,1 \,\leq\, i \,\leq\, d\,).
\end{equation} 
Therefore,
\[\left\|\,x_{\,k} \,-\, x \,,\, b_{\,2} \,,\, \cdots \,,\, b_{\,n}\,\right\| \,=\, \left\|\,\sum\limits_{i \,=\, 1}^{\,d}\, \left(\,a_{k \,,\, i} \,-\, a_{\,i}\,\right)\, e_{\,i} \,,\, b_{\,2} \,,\, \cdots \,,\, b_{\,n}\,\right\|\]
\[\hspace{4.2cm}\leq\; \sum\limits_{i \,=\, 1}^{\,d}\, \left|\,a_{k \,,\, i} \,-\, a_{\,i}\,\right|\, \left\|\,e_{\,i} \,,\, b_{\,2} \,,\, \cdots \,,\, b_{\,n}\,\right\|\]
\[\hspace{2.6cm} \to\; 0\; \;\text{as}\; \;k \;\to\; \infty\;  \;[\;\text{by}\; \;(\,\ref{eq3}\,)\;]\]
\[\Rightarrow\; \lim\limits_{k \;\to\; \infty}\, \left\|\,x_{\,k} \,-\, x \,,\, b_{\,2} \,,\, \cdots \,,\, b_{\,n}\,\right\| \,=\, 0 \] and hence \,$\{\,x_{\,k}\,\}$\; converges \,$b$-strongly to \,$x$\; in \,$X$.
\end{proof}

\end{document}